\documentclass[11pt]{article}

\usepackage[utf8]{inputenc}
\usepackage{amsthm,amsmath,amsfonts,latexsym,url}
\usepackage{amssymb}
\usepackage{comment}
\usepackage{mathrsfs}
\usepackage[colorlinks=true,linkcolor=blue,citecolor=blue,urlcolor=magenta,pdfpagelabels=false]{hyperref}
\usepackage{xcolor}
\usepackage{enumitem}
\usepackage{tikz}
\usepackage[margin=3cm]{geometry}
\usepackage{charter}
\usepackage{dsfont}
\usepackage[all]{xy}
\usepackage{microtype}
\usepackage{mathtools}

\usepackage[capitalize]{cleveref}
\usepackage[english]{babel}
\usepackage{bigints}

\usepackage{scalerel,stackengine}
\stackMath
\newcommand\reallywidehat[1]{%
\savestack{\tmpbox}{\stretchto{%
  \scaleto{%
    \scalerel*[\widthof{\ensuremath{#1}}]{\kern-.6pt\bigwedge\kern-.6pt}%
    {\rule[-\textheight/2]{1ex}{\textheight}}
  }{\textheight}%
}{0.5ex}}%
\stackon[1pt]{#1}{\tmpbox}%
}

\newtheorem{theo}{Theorem}[section]

\newtheorem{coro}[theo]{Corollary}

\theoremstyle{definition}

\newcommand{\Q}{\mathbb{Q}}

\usepackage[symbol]{footmisc}

\def\pFqnoargs#1#2{{}_#1F_#2}
\def\pFq#1#2#3#4#5{
  \mathchoice
      {\pFqnoargs{#1}{#2}\biggl(\begin{matrix}{\def,{\kern.707em}#3}\\{\def,{\kern.707em}#4}\end{matrix}\,\bigg|\,#5\biggr)} 
      {\pFqnoargs{#1}{#2}(#3;#4;#5)} 
      {\pFqnoargs{#1}{#2}(#3;#4;#5)} 
      {\pFqnoargs{#1}{#2}(#3;#4;#5)} 
}
\def\twoFone#1#2#3#4{\pFq21{#1,#2}{#3}{#4}}

\def\twoFone#1#2#3#4{{_2F_1}\biggl(\begin{matrix}
  {#1}\kern.707em {#2}\\{#3}
\end{matrix}\,\bigg|\,#4\biggr)}

\newcommand{\footremember}[2]{%
    \footnote{#2}
    \newcounter{#1}
    \setcounter{#1}{\value{footnote}}%
}

\usepackage{graphicx}

\title{An arithmetic characterization of some algebraic functions and a new proof of an algebraicity prediction by Golyshev}

\author{%
  \href{https://mathexp.eu/bostan/}{Alin Bostan}\footremember{1}{Inria, Université Paris-Saclay, 1 rue Honoré d'Estienne d'Orves, 91120 Palaiseau, France (\href{mailto:alin.bostan@inria.fr}{alin.bostan@inria.fr}).} 
}

\date{December 28, 2024}

\begin{document}

\maketitle

\begin{abstract}
We provide a new arithmetic characterization for the sequence of coefficients of algebraic power series $f(t)$ having the property that the differential equation $y'(t) = f(t) y(t)$ has algebraic solutions only.
This extends a recent result by Delaygue and Rivoal, and also provides a new and shorter proof of an algebraicity result predicted by Golyshev.
\end{abstract}

\section{Introduction} \label{sec:intro}
Let $\ell \in \mathbb{N}$ and assume that ${\bf c} = (c_1, \ldots, c_\ell) \in \mathbb{N}^\ell$ and ${\bf d} = (d_1, \ldots, d_{\ell+1}) \in \mathbb{N}^{\ell+1}$ are tuples of nonnegative integers such that 
$\sum_i c_i = \sum_j d_j$ and such that 
the sequence $(a_n)_{n \geq 0}$ defined by
\begin{equation} \label{def:an}
a_n \coloneqq \frac{(c_1 n)!\cdots (c_\ell n)!}{(d_1 n)!\cdots (d_{\ell+1} n)!}
\end{equation}
is an integer sequence. The generating function
\[
F(t)
\coloneqq \sum_{n=0}^\infty a_n t^n
\]
is then hypergeometric and belongs to $\mathbb{Z}[[t]]$.
Building on the ``interlacing criterion'' of Beukers and Heckman~\cite{BeHe89}, Rodriguez-Villegas proved in~\cite{Villegas05} that $F$ is algebraic over $\mathbb{Q}(t)$, i.e., there exists a 
polynomial $P \in \mathbb{Q}[x,y] \setminus \{ 0 \}$ such that  $P(t,F(t)) = 0$.
The simplest example is 
\[F(t) = \sum_{n=0}^\infty \binom{cn}{dn} t^n, \qquad \text{with} \quad \ell=1, \; c_1 = c, \; d_1=d, \; d_2=c-d, \] 
where $c \geq d \geq 0$.
More interesting examples (see also~\cite{Bober09}) are given by the power series
\[
\sum_{n=0}^\infty \frac{(6n)!n!}{(3n)!(2n)!^2} t^n, \quad
\sum_{n=0}^\infty \frac{(10n)!n!}{(5n)!(4n)!(2n)!} t^n \quad \text{and} \quad
\sum_{n=0}^\infty \frac{(30n)!n!}{(15n)!(10n)!(6n)!} t^n.
\]
By Rodriguez-Villegas' result, these three power series are all algebraic. It can be proved (for instance by using a computer algebra system) that the algebraicity degrees of the first two ones are 6 and 30, while the degree of the third one is 483{\,}840 according to~\cite{Villegas05,Villegas19}.

As accounted in Zagier's beautiful survey~\cite[p.~757--759]{Zagier18}, Golyshev predicted (based on ``motivic arguments'') that the power series
\begin{equation} \label{def:Q}
Q(t) \coloneqq \exp \left( \int \frac{F(t)}{t}  dt \right) = t \exp \left( \sum_{n>0} a_n \frac{t^n}{n}\right)
\end{equation}
is also an algebraic function. 
This ``prediction'' has been recently proved by Delaygue and Rivoal in~\cite{DeRi23}.
Their proof consists of two steps: 

\; (i) the sequence $(a_n)_{n \geq 0}$ defined in~\eqref{def:an} has the \emph{Gauss property}; this means that, for almost all prime numbers~$p$, the Gauss congruences hold: $\frac{1}{n} \cdot \left( a_{np} - a_n \right) \in p \mathbb{Z}_{(p)}$ for all $n\geq 1$\footnote[1]{Here and in all the text, we denote by $\mathbb{Z}_{(p)}$ the localization of $\mathbb{Z}$ at the prime ideal $(p)$, that is the subring of~$\Q$ consisting of rational numbers with denominator coprime to $p$.};

\; (ii) clearly, $y(t) \coloneqq Q(t)/t$ satisfies 
$y'(t) = f(t) y(t)$, where
$f(t) \coloneqq \sum_{n>0} a_n t^{n-1}$; from this deduce that, as $(a_n)_{n \geq 0}$ has the Gauss property, $y(t)$ must be algebraic (and hence $Q(t)$ too).

Both steps are nontrivial. 
Step~(i) is based on~\cite[Lemme~10]{Delaygue12}, itself based on quite subtle $p$-adic congruences for Morita's $p$-adic Gamma functions. 
Step~(ii) is based on the one hand on a consequence of Theorem 8.1 in~\cite{ChCh85} saying that if $g \in \mathbb{Z}[[\frac{t}{\lambda}]]$ for some $\lambda\in \mathbb{N}\setminus \{ 0 \}$ is such that $g'/g$ is algebraic, then $g$ itself is algebraic; on the other hand, it is based on the classical Dieudonné-Dwork lemma~\cite[p.~392]{Robert00} as well as on a refined $p$-adic analysis
(\cite[\S3.2.1]{DeRi23}) showing that $Q(t)$ defined in~\eqref{def:Q}
belongs to $\mathbb{Z}[[\frac{t}{\lambda}]]$ for some $\lambda\in \mathbb{N}\setminus \{ 0 \}$.

\section{A new proof of Golyshev's prediction} \label{sec:proof}

In this note we provide a similar, but shorter, proof of the algebraicity of $Q(t)$ in~\eqref{def:Q}. We  directly use~\cite[Theorem 8.1]{ChCh85} 
which gives a positive answer to Grothendieck’s conjecture for linear differential equations of order 1 with algebraic coefficients\footnote[2]{Another proof of Thm.~8.1 in \cite{ChCh85} is given in \cite[Chap.~VIII, Exs.~5 and 6]{Andre89}. The result was generalized by André to linear differential equations whose differential Galois groups have a solvable neutral component in~\cite[Cor.~4.3.5]{Andre04} (see also \cite[Thm.~2]{Andre87} and \cite[Thm.~3.5]{ChambertLoir02}). It was further extended by Bost \cite[Thm.~2.9]{Bost01} who proved a generalized version of the Grothendieck-Katz conjecture for any algebraic group (not necessarily linear) over an number field, whose neutral component is solvable. To our knowledge, all these proofs ultimately rely on algebraicity criteria generalizing the classical Borel-Dwork rationality criteria~\cite[Thms. 2 and 3]{Dwork60}, originally due to the Chudnovsky brothers~\cite[Thm.~5.9]{ChCh85} and further extended by André and Bost in \cite[Chap.~VIII, Thms.~1.2 and 1.4]{Andre89}, \cite[Thm.~5.4.3 and Cor.~5.4.5]{Andre04}, \cite[Thms.~2.1 and 3.4]{Bost01}, see also \cite[Thms. 4.3 and 6.2]{ChambertLoir02}.}.
Compared to the proof in~\cite{DeRi23} sketched above, we do not need Gauss congruences and we avoid the use of \cite[Lemme~10]{Delaygue12}, of the Dieudonné-Dwork lemma, and of the $p$-adic argument in \cite[\S3.2.1]{DeRi23}.

Instead, we prove in a direct and easy way that the sequence $(a_n)_{n \geq 0}$ satisfies the weaker congruences
\[
 a_{np} - a_n \in p \mathbb{Z}
  \quad \text{for all} \; n\geq 1 ,
 \]
and then conclude by using an explicit computation of the $p$-curvatures of $y'(t) = f(t) y(t)$.

\begin{theo}\label{thm:1}
Let $(a_n)_{n \geq 1}$ be an integer sequence 
and assume that  $f(t) \coloneqq \sum_{n>0} a_n t^{n-1}$ is algebraic over $\mathbb{Q}(t)$.
Then, all solutions of the differential equation $y'(t) = f(t) y(t)$ are algebraic 
if and only if, for almost all prime numbers~$p$,
the following infinite system of congruences holds
\begin{equation} \label{eq:congruences}
a_{np} \equiv a_{n}\pmod p \quad \text{for all} \; n \geq 1.
\end{equation}
\end{theo}

\begin{proof}
By \cite[Theorem 8.1]{ChCh85}, all solutions of $y'(t) = f(t) y(t)$ are algebraic if and only if for almost all primes~$p$, its $p$-curvatures are zero.
To finish the proof, it is enough to show that, for any prime~$p$, the vanishing of the $p$-curvature is equivalent to congruences~\eqref{eq:congruences}.
It is known since Jacobson's paper~\cite{Jacobson37} that the $p$-curvature of $y'(t) = f(t) y(t)$ is equal (up to a sign) to $f^{(p-1)}(t) + f(t)^p$ modulo~$p$.
Therefore, its vanishing is equivalent to 
$f(t)^p = - f^{(p-1)}(t) \pmod p$.
To conclude, it is enough to remark that
$(f^{(p-1)}(t))^{1/p} = -\sum_{n \geq 1} a_{np} t^{n-1}$.
Indeed, 
\[
f^{(p-1)}(t) = \sum_{\ell \geq p} a_\ell (\ell-1)(\ell-2) \cdots (\ell-(p-1)) t^{\ell-p}
= - \sum_{n \geq 1} a_{np}  t^{(n-1)p} \pmod p, 
\]
since the product $(\ell-1)(\ell-2) \cdots (\ell-(p-1))$ is nonzero modulo~$p$ if and only if $p$ divides $\ell$, in which case it is equal to $-1$ modulo~$p$, by Wilson's theorem.
\end{proof}

\begin{theo}\label{thm:2}
Assume that
$(a_n)_{n \geq 1}$ defined by~\eqref{def:an} is an integer sequence with $\sum_i c_i = \sum_j d_j$.
Then, for all prime numbers~$p$, the sequence 
$(a_n)_{n \geq 1}$ satisfies the  
system of congruences~\eqref{eq:congruences}.
\end{theo}

\begin{proof}
We start by observing that for any prime~$p$ and any $m\in\mathbb{N}$, there exists $\alpha_m\in\mathbb{Z}$ with
\[
(mp)! = m! (-p)^m (1 + p \alpha_m) .
\]
This is again a simple consequence of Wilson's theorem: 
indeed, 
$(mp)!= \left( 1\cdots (p-1) \right) \cdot  p 
\cdot \left( (p+1)\cdots (2p-1) \right) \cdot (2p)
\cdots 
\left( ((m-1)p+1)\cdots (mp-1) \right) \cdot (mp)
$
is equal to 
$
p \cdot (2p) \cdots (mp) \cdot (p s_1 - 1) \cdots (p s_m-1)
$
for some $s_i\in\mathbb{Z}$, which rewrites 
$m! \cdot (-p)^m \cdot (1 + p \alpha_m)$ for some $\alpha_m \in \mathbb{Z}$.

It follows that there exist integer numbers $\alpha_{c_i n}$ and $\alpha_{d_j n}$ such that
\[
a_{np} 
= \frac{\prod_{i=1}^\ell (c_i np)!}{\prod_{j=1}^{\ell+1} (d_j np)!}
= \frac{\prod_{i=1}^\ell (c_i n)! (-p)^{c_i n} (1 + p \alpha_{c_i n})}{\prod_{j=1}^{\ell+1} (d_j n)! (-p)^{d_j n} (1 + p \alpha_{d_j n})}
= 
a_{n} (-p)^{(\sum_i c_i - \sum_j d_j)n} \frac{\prod_{i=1}^\ell (1 + p \alpha_{c_i n})}{\prod_{j=1}^{\ell+1} (1 + p \alpha_{d_j n})} .
\]
Since $\sum_i c_i = \sum_j d_j$ and the last products belong to $1+p\mathbb{Z}_{(p)}$, the conclusion follows.
\end{proof}

By combining \cref{thm:1} and \cref{thm:2}, we deduce that Golyshev's prediction holds.

\begin{coro}\label{coro}
Let $F(t)$ be an algebraic factorial generalized hypergeometric function, that is, a power series of the form
\[
F(t)
\coloneqq \sum_{n=0}^\infty \frac{(c_1 n)!\cdots (c_\ell n)!}{(d_1 n)!\cdots (d_{\kappa} n)!} t^n \quad (c_i, d_j \in \mathbb{N})
\]
assumed to be algebraic over $\mathbb{Q}(t)$.
Then, 
$
\exp \left( \int \frac{F(t)}{t}  dt \right)
$
is also algebraic over $\mathbb{Q}(t)$.
\end{coro}

\begin{proof}
By (the easy direction of) \cite[Theorem 1]{Villegas05}, the algebraicity of $F$ implies that its sequence of coefficients $(a_n)$ is an integer sequence, that $\kappa=\ell + 1$ and $\sum_i c_i = \sum_j d_j$. 
In other words, $(a_n)$ is of the form~\eqref{def:an} and satisfies the assumptions of~\cref{thm:2}. 
By~\cref{thm:2}, $(a_n)$ satisfies the system of congruences~\eqref{eq:congruences} for all prime numbers~$p$. 
On the other hand, $f(t) \coloneqq (F-1)/t = \sum_{n>0} a_n t^{n-1}$ is algebraic, as $F$ is algebraic.
Now, $y\coloneqq \exp \left( \int \frac{F(t)}{t}  dt \right)/t$ satisfies the differential equation
$y'(t) = f(t) y(t)$.
The conclusion follows from~\cref{thm:1}.
\end{proof}

We now state an alternative form of Golyshev's prediction.

\begin{coro}\label{coro:2}
Let $\ell \in \mathbb{N}$ and assume that $(c_1, \ldots, c_\ell) \in \mathbb{N}^\ell$ and $(d_1, \ldots, d_{\ell+1}) \in \mathbb{N}^{\ell+1}$ are such that $\sum_i c_i = \sum_j d_j$ and such that 
	the sequence $(a_n)_{n \geq 0}$ defined by
\[	a_n \coloneqq \frac{(c_1 n)!\cdots (c_\ell n)!}{(d_1 n)!\cdots (d_{\ell+1} n)!} \]
is an integer sequence. 
Then the power series
	\[
	 \sum_{n=0}^\infty a_n t^n \quad \text{and} \quad 
	\exp \left( \sum_{n=1}^\infty a_n \frac{t^n}{n}\right)
	\]
are  algebraic over $\mathbb{Q}(t)$.
\end{coro}

\begin{proof}
By (the difficult direction of) \cite[Theorem 1]{Villegas05}, $\sum_{n=0}^\infty a_n t^n$ is algebraic. The proof heavily relies on the interlacing criterion due to Beukers and Heckman~\cite{BeHe89}.
The algebraicity of $\exp \left( \sum_{n=1}^\infty a_n \frac{t^n}{n}\right)$ can then be proved as in the second part of the proof of~\cref{coro}.
\end{proof}

\section{Main result} \label{sec:main}

We now prove a more general version of~\cref{thm:1}. 
This is the main result of this note.
Not only we relax the integrality assumption on the sequence $(a_n)$, but we prove the equivalence of the system of congruences~\eqref{eq:congruences} for almost all primes~$p$ with the Gauss congruences considered in~\cite{DeRi23} for almost all primes~$p$.

\begin{theo}\label{thm:1gen}
Let $(a_n)_{n \geq 1}$ be a sequence of rational numbers
and assume that  $f(t) \coloneqq \sum_{n>0} a_n t^{n-1}$ is algebraic over $\mathbb{Q}(t)$.
Then, the following assertions are equivalent:
\begin{itemize}
	\item[(a)] all solutions of the differential equation $y'(t) = f(t) y(t)$ are algebraic over $\mathbb{Q}(t)$;
	\item[(b)] for almost all prime numbers~$p$, we have that
$a_{np} - a_{n} \in p \mathbb{Z}_{(p)}$ for all $n \geq 1$;
	\item[(c)] for almost all prime numbers~$p$, we have that
$a_{np} - a_{n} \in n p \mathbb{Z}_{(p)}$ for all $n \geq 1$.
\end{itemize}
\end{theo}

Condition~(\emph{c}) is called the \emph{Gauss property} for the sequence $(a_n)$ (or, equivalently, for the power series $t f(t)$).

Note that the same statement holds true if $f(t)$ is assumed to belong more generally to $\mathbb{Q}((t))$, and to be algebraic over $\mathbb{Q}(t)$. Indeed, with this assumption one can show as in \cite[\S3.2]{DeRi23} that if $y'(t) = f(t) y(t)$ has only algebraic solutions, then $t f(t)$ belongs to $\mathbb{Q}[[t]]$; similarly, if $(a_n)_{n > b}$ ($b\in\mathbb{Z}$) satisfies the Gauss property, then $a_n = 0$ for $n<0$; and the same holds for the congruences in condition~(\emph{b}). In other words, proving~\cref{thm:1gen} with the assumption $f\in\mathbb{Q}((t))$ reduces easily to the case $\mathbb{Q}[[t]]$ treated in~\cref{thm:1gen}. 

\begin{proof}
Obviously, (\emph{c}) implies (\emph{b}), and 
Eiseinstein’s theorem together with the Dieudonné-Dwork criterion show that (\emph{a}) implies~(\emph{c}) (see \cite{DeRi23}).
It remains to prove that (\emph{b}) implies~(\emph{a}). 
This was proved in our~\cref{thm:1} under the stronger assumption that $f\in\mathbb{Z}[[t]]$. As $f$ is algebraic, Eisenstein's theorem implies that there exists $c\in\mathbb{N} \setminus \{ 0 \}$ such that $g(t)\coloneqq c f(tc)$ belongs to $\mathbb{Z}[[t]]$. In other words, the sequence $b_n \coloneqq c^{n+1} a_n$ is an integer sequence. We shall prove that $b_{np} \equiv b_n \pmod p$
for almost all primes~$p$ and for all $n \geq 1$.
Let $B\geq 1$ be such that $a_{np} - a_{n} \in p \mathbb{Z}_{(p)}$ for all $n \geq 1$ and all primes $p \geq B$. Let $B'\geq B$ be such that all prime divisors of~$c$ are less than $B'$. 
Then, for all primes $p \geq B'$ we have that
$b_{np} - b_n = c^{np+1} a_{np} - c^{n+1} a_n = c( c^{np} - c^n) a_{np} + c^{n+1} (a_{np} - a_n)$
belongs to $p \mathbb{Z}_{(p)}$, because $p$ divides the integer $c^{np} - c^n$, while $a_n \in \mathbb{Z}[\frac{1}{c}] \subset \mathbb{Z}_{(p)}$ and $a_{np} - a_n \in p \mathbb{Z}_{(p)}$.
Thus, $b_{np} - b_n \in \mathbb{Z} \cap p \mathbb{Z}_{(p)} = p \mathbb{Z}$ for all primes $p\geq B'$ and for all $n\geq 1$. By~\cref{thm:1}, all solutions of the differential equation $y'(t) = g(t) y(t)$ are algebraic. As $g(t)= c f(tc)$, this implies that all solutions of the differential equation $y'(t) = f(t) y(t)$ are algebraic as well.
\end{proof}

\section{Conclusion and further questions}

A surprising consequence of our results is the equivalence between conditions (\emph{b}) and (\emph{c}) in~\cref{thm:1gen} when each holds \emph{for almost all primes~$p$}. This equivalence is a priori unexpected, and it is in fact the main point in our simpler proof in~\cref{sec:proof}: to prove Golyshev's prediction, instead of checking the Gauss property for the factorial ratios sequences, it is enough to check the more relaxed congruences~\eqref{eq:congruences}.
Let us remark that (\emph{b}) and (\emph{c}) are in general not equivalent for one prime~$p$.
On the other hand, the equivalence between (\emph{a}) and (\emph{b}) does not hold if one replaces ``for almost all primes $p$'' by ``for an infinite number of primes~$p$'': for instance, when $f(t)=1/(t^2+1)$, (\emph{a}) does not hold, but (\emph{b}) and (\emph{c}) hold for all primes~$p$ equal to $1$ modulo~4.

Independently of the algebraicity context of~\cref{thm:1gen}, it is legitimate to ask for which integer (or, globally bounded) sequences $(a_n)$ does the equivalence between (\emph{b}) and (\emph{c}) hold?
For instance, does it hold for (some of) the integer factorial ratios sequences of the form
$
a_n = \prod_{i=1}^\ell (c_i n)!/\prod_{j=1}^\kappa (d_j n)!
$
studied by Soundararajan in~\cite{Soundararajan22}? (See also~\cite{DvZa01}.)

Another natural question is whether there exists a ``finite version'' of~Theorems~\ref{thm:1} and~\ref{thm:1gen}. By this, we mean: given the algebraic function~$f(t)$, can one find effective bounds $B_1(f)$ and $B_2(f)$ such that one can replace conditions (\emph{b}) and (\emph{c}) in \cref{thm:1gen} by their finite versions:
\begin{itemize}
	\item[($b'$)] 
$a_{np} - a_{n} \in p \mathbb{Z}_{(p)}$ for all $n \geq 1$ and all primes $p\in [B_1(f), B_2(f)]$;
	\item[($c'$)] 
$a_{np} - a_{n} \in np \mathbb{Z}_{(p)}$ for all $n \geq 1$ and all primes $p\in [B_1(f), B_2(f)]$.
\end{itemize}
As the proof of~\cref{thm:1gen} is ultimately based on Chudnovskys' Theorem 8.1 in~\cite{ChCh85}, this question amounts to asking whether there exists a finite version of that theorem.

Another natural question is whether~\cref{coro} and~\cref{coro:2} can be proved without relying on Chudnovskys' theorem at all. For instance, can it be deduced using the notion of $p^i$-curvatures studied in the recent preprints~\cite{FuHaKa24} and \cite{Movasati24}?

Finally, one may wonder whether more explicit structural information can be obtained in the conclusion of~\cref{coro}, e.g. concerning the algebraicity degrees. For instance, can they be expressed in terms of the numbers $c_i$ and $d_j$? According to Zagier~\cite[p.~757]{Zagier18}, Golyshev had also predicted that, under the assumptions of~\cref{coro}, the power series $\exp \left( \sum_{n=1}^\infty a_n {t^n}/{n}\right)$ has the same algebraicity degree as $\sum_{n=0}^\infty a_n t^n$, and moreover that it is an algebraic unit over $\mathbb{Z}[1/t]$. To our knowledge, all these are open questions.

\medskip 
\noindent \thanks{{\bf Acknowledgements}. 
This work was supported by the French grant DeRerumNatura (ANR-19-CE40-0018) and the French--Austrian project EAGLES (ANR-22-CE91-0007 \& FWF I6130-N).
Warm thanks go to Xavier Caruso for enlightening discussions, to Éric Delaygue and Sergey Yurkevich for their very careful reading of preliminary versions of this text, and to Gilles Christol, Vasily Golyshev, Maxim Kontsevich, Hossein Movasati, Umberto Zannier, Wadim Zudilin and the two anonymous reviewers for their positive feedback. 
}

\bibliographystyle{alpha}

\end{document}